\documentclass{amsart}
\usepackage[utf8x]{inputenc}
\usepackage{bbm}
\usepackage{amsmath,amsthm,amsfonts,amssymb,amscd,mathrsfs}
\usepackage{psfrag,graphicx}

\usepackage{epic,eepic}
\usepackage{color}
\usepackage{ebezier}
\usepackage{enumerate}

\newtheorem{Thm}{Theorem}[section]
\newtheorem{Lem}[Thm]{Lemma}
\newtheorem{Prop}[Thm]{Proposition}
\newtheorem{Cor}[Thm]{Corollary}
\newtheorem{Def}[Thm] {Definition}
\newtheorem{Rem}[Thm] {Remark}

 \thanks{ $^{*}$School of Mathematical Sciences,
Peking University, Beijing 100871, China; 
}
 
\title{Horseshoes for  $\mathcal{C}^{1+\alpha}$ mappings with hyperbolic measures}
\date{Aug 8, 2014}
\author{Yun Yang$^{*}$}

\begin{document}

\maketitle
\begin{abstract}We present here a construction of  horseshoes for any $\mathcal{C}^{1+\alpha}$ mapping $f$ preserving an ergodic hyperbolic measure $\mu$ with $h_{\mu}(f)>0$ and then deduce that
 the exponential growth rate of the number of periodic points for any $\mathcal{C}^{1+\alpha}$ mapping $f$
is greater than or equal to $h_{\mu}(f)$.  We also prove that the exponential growth rate of the number of hyperbolic periodic points  is equal to the hyperbolic entropy. The hyperbolic entropy means the entropy resulting from hyperbolic measures. 
\end{abstract}
\section{Introduction}
In this paper, we build horseshoes for $\mathcal{C}^{1+\alpha}$ mappings (not necessarily invertible) preserving ergodic hyperbolic measures with positive measure-theoretical  entropy and then prove that the exponential growth rate of the number of periodic points 
is greater than or equal  to the measure-theoretical entropy.  
This research  is a  natural generalization
 of Katok's argument in his paper \cite{Kat}. We also prove that the exponential growth rate of the number of hyperbolic periodic points with ``large" Lyapunov exponents  is equal to hyperbolic entropy. Hyperbolic entropy means the entropy resulting from hyperbolic measures. 
 
Horseshoes are exhibited as examples of systems that demonstrate complicated dynamical behaviors and allow us to model the behavior by a shift map over a finite alphabet. 
Thus, it is an interesting problem to consider the existence of horseshoes.  Let $M$ be a compact  manifold of dimension 2 and $f:M\rightarrow M$ be a $\mathcal{C}^{1+\alpha}$ diffeomorphism with positive entropy.
 Katok's argument  illustrates the fact that  positive entropy implies the existence of  horseshoes  
and  the entropy of these inner horseshoes can approximate $h_{\mu}(f)$, which means the underlying horseshoes demonstrate nearly the same complicated property as the whole  systems.    One might expect Katok's argument to be true for endomorphsims with all Lyapunov exponents not zero.  

Gelfert\cite{Gelfert} proved the existence of horseshoe for mappings with only positive Lyapunov exponents under some integrability conditions that are used to control the effect of critical points and singular points. We give a generalization from the nonsingular case to  all Lyapunov exponents not zero, without integrability assumption on  critical points (Theorem \ref{main theorem}). Besides, we also control the Lyapunov exponents of periodic points in the horseshoe. 
After completing this paper, we came upon a paper by Y.M. Chung \cite{Chung} who dealt   with the same problem, but the starting point of our proof is shadowing lemma for sequences maps which is  different from the idea used in the proof of \cite{Chung}.  Also, \cite{Chung} does not give results on controlling the Lyapunov exponents of periodic points which is used in the proof of Theorem \ref{hyperbolic growth} in this paper.

Now let us state our main results. Let $M$ be closed $d-$dimensional Riemannian manifold.  
\begin{Def}For any continuous map $T$ on  metric space $N$,  {\bf the inverse limit space} $N^T$ of $(N, T)$ is the subset of $N^{\mathbb{Z}}$ consisting of all full orbits, i.e.
$$N^{T}=\{\tilde{x}=(x_i)_{i\in \mathbb{Z}}|x_i\in N, Tx_i=x_{i+1}, \forall i\in\mathbb{Z}\}.$$
\end{Def}
There exists a natural metric defined as 
$$d(\tilde{x},\tilde{y})=\sum_{i=-\infty}^{i=+\infty}\frac{d(x_i,y_i)}{2^{|i|}}.$$ Thus $N^T$ is a metric space with norm  satisfing $\max_{i}d(x_i,y_i)\geq d(\tilde{x},\tilde{y})\geq d(x_0,y_0)$.  Let $\tilde{T}$  be the shift map $\tilde{T}((x_i)_{i\in\mathbb{Z}})=(x_{i+1})_{i\in\mathbb{Z}}$  on $N^{T}$. From Lemma \ref{EM}, the set of invariant measures of $(\tilde{T}, N^T)$ and the set of invariant measures of $(T,N)$ are equivalent. Denote $\tilde{\mu}$ as the extension measure for $\mu$. This extension also keeps entropy, i.e. $h_{\tilde{\mu}}(\tilde{T})=h_{\mu}(T).$

\begin{Def}Fix a  continuous map $T$ on a metric space $N$. We say that $T:N\rightarrow N$ has a {\bf topological horseshoe} if there exists a $T$-invariant compact set $\Lambda$ such that the restriction of $T$ on $\Lambda$ is topologically conjugate to a subshift of finite type $\sigma:\mathcal{A}^{\mathbb{Z}}\rightarrow \mathcal{A}^{\mathbb{Z}}$. 
\end{Def}

\begin{Thm}\label{main theorem}
Let $f:M\rightarrow M$ be a $\mathcal{C}^{1+\alpha}$ mapping  preserving an ergodic
hyperbolic probability measure $\mu$ with entropy $h_{\mu}(f)>0$ and let $\tilde{\mu}$ be the extension measure of $\mu$ to the inverse limit space $M^f$.   For any constant $\delta>0$  and a weak $*$ neighborhood $\tilde{\mathcal{V}}$ of $\tilde{\mu}$ in the space of $\tilde{f}$-invariant probability measures, there exists a horseshoe $\tilde{H}\subset M^f$ such that:
\begin{enumerate}
 \item $h_{top}(\tilde{H},\tilde{f})>h_{\tilde{\mu}}(\tilde{f})-\delta=h_{\mu}(f)-\delta$. 
\item if $\tilde{\lambda}_1>\tilde{\lambda}_2>\cdots>\tilde{\lambda}_k$ are the distinct Lyapunov exponents of $\mu$, with 
multiplicities $n_1,\cdots,n_k\geq 1$, denote $\tilde{\lambda}$ the same as before, then there exists a dominated splitting on $T_{\tilde{x}}M=\sqcup T_{\pi(\tilde{f}^n\tilde{x})}M, \tilde{x}\in \tilde{H}$ where $\sqcup$ means the disjoint union,
$$T_{\tilde{x}}M=E^u\oplus E^s,$$ and  there exists $N\geq1$ such that for each $i=1,2$ each $\tilde{x}\in\tilde{H}$ and each unit vector $v\in E^u(\pi(\tilde{x})), u\in E^s(\pi(\tilde{x}))$,
$$||Df^{-N}_{\pi(\tilde{x})}(v)||\leq \exp((-\tilde{\lambda}_i+\delta)N),$$$$||Df^{N}_{\pi(\tilde{x})}(u)||\leq \exp((-\tilde{\lambda}_i+\delta)N)$$
\item all the invariant probability measures supported on $\tilde{H}$ lies in $\mathcal{V}.$

\item $\tilde{H}$ is $\delta-$close to the support of $\tilde{\mu}$  in the Hausdorff distance.
\end{enumerate}

\end{Thm}
\begin{Cor}Let $f:M\rightarrow M$ be a  $\mathcal{C}^{1+\alpha}$ mapping   preserving an ergodic
hyperbolic probability measure $\mu$.  We have 
$$\limsup_{n\rightarrow+\infty}\frac{1}{n}\log P_n(f)\geq h_{\mu}(f)$$
where $P_n(f)$ denotes the number of periodic points with  period $n$. 
\end{Cor}
Theorem \ref{hyperbolic growth} below is a generalization of a result by Chung and Hirayama \cite{Chung2}. They proved that the topological entropy of a $\mathcal{C}^{1+\alpha}$ surface diffeomorphism is given by the growth rate of the number of periodic points of saddle type.  We prove here that for any $\mathcal{C}^{1+\alpha}$ mappings on any dimensional manifold, the growth rate of the number of hyperbolic periodic points equals to  the entropy coming from hyperbolic measures, hyperbolic entropy (see Defnition \ref{hyperbolic entropy}).  

We  point out that there is a similar  result concerning  the topological pressure for diffeomorphisms in Gelfert and Wolf's paper \cite{GW2}.  They proved that, for $\mathcal{C}^{1+\alpha}$ diffeomorphisms,  topological pressure for potentials with only hyperbolic equilibrium states is totally determined by  the value of potentials on saddle periodic points with ``large" Lyapunov exponents. 
\begin{Def}\label{hyperbolic entropy}Let $f:M\rightarrow M$ be a  $\mathcal{C}^{1+\alpha}$ mapping.  Let $\mathcal{HM}$ be the set of hyperbolic ergodic invariant measures of $f$ and let  $H(f)=\sup_{\mu\in \mathcal{HM}}h_{\mu}(f).$ We call $H(f)$ the {\bf hyperbolic entropy} of $f$.
\end{Def}
For surface diffeomorphisms, all invariant measures with positive entropy are hyperbolic. So, hyperbolic entropy equals to topological entropy for surface diffeomorphisms. \begin{Thm}\label{hyperbolic growth}Let $f:M\rightarrow M$ be a  $\mathcal{C}^{1+\alpha}$ mapping on a closed Riemannian manifold $M$. We have 
$$\limsup_{n\rightarrow+\infty}\frac{1}{n}\log P_n(f)\geq H(f).$$
 Moreover,   we have $$\lim_{a\rightarrow 0^{+}}\lim_{K\rightarrow 0^{+}}\limsup_{n\rightarrow \infty}\frac{1}{n}\log \sharp PH(n,f, K,a)=H(f).$$
where $PH(n,f,K,a)$ means the number of  collection of periodic points with $n$ period and with uniform $(K,a)$-hyperbolicity (see Definition \ref{PH}). 
\end{Thm}

Now we give a short discussion about the main techniques we used in this paper. 
As stated before, the starting point of our proof of Theorem \ref{main theorem} is the shadowing lemma for sequences of mappings.
In Avila, Crovisier and Wilkinson's new paper\cite{ACW},  they give a compact proof of shadowing lemma using the shadowing lemma for sequences of  mappings and then they establish a direct way to find a horseshoe by coding  some special separated set directly. 
Inspired by their ideas, we  establish a shadowing property for extension map $\tilde{f}$ in the inverse limit space $M^{f}$ which  inherits many properties for the mapping $f$ 
 and then construct horseshoes in the inverse limit space. 


Finally, we end  this section with  a short note about  critical points.  The key issue  caused  by critical points is the switch of the unstable direction and the stable direction which will cancel the hyperbolicity.  Such occurs, for example, in  snap-back repellers.  
Such phenomenon highly conflicts with the existence of  absolutely continuous invariant measures (acim).  Thus there are many compositions,  which concern the existence of acim,   taking  critical points into account carefully (\cite{ABV,Led,BM} ).  Nevertheless, by Ma\~{n}e's multiplicative ergodic theorem, the derivatives along  the unstable directions of almost all orbits in a Pesin block are isomorphisms. 
So, in terms of  the shadowing lemma and the construction of horseshoes,  the only collapse may happen here is along the stable direction which does not affect the shadowing lemma.    



\section*{Acknowledgement}I  would like  to thank  Amie Wilkinson for fruitful conversations. She is not only a  research mentor  but also a spiritual mentor to me. I would also like to thank Zhihong Xia for useful conversations  in the preparation of this paper.   This work was done during my stay in The University of Chicago. I would like to thank  The University of Chicago for their hospitality and China Scholarship Council  for their financial support.
\section{Preliminaries}
\subsection{Inverse limit space}
First we give the definition of regular Anosov mappings to illustrate some differences between diffeomorphisms and  mappings in dynamical features.

\begin{Def}\cite{Przytycki}A regular  map $f\in \mathcal{C}^1
(M,M)$ is an {\bf Anosov mapping} if there exist constants $C>0,0<\lambda<1$
and a Riemannnian metric $<\cdot,\cdot>$ on $TM$ such that for every $f$-orbit $(x_n)_{n\in\mathbb{Z}}\in M^{f}$, there is a  splitting of 
$\sqcup^{+\infty}_{-\infty}T_{x_n}M=E^s_{(x_n)_{n\in\mathbb{Z}}}\oplus E^u_{(x_n)_{n\in\mathbb{Z}}}=\sqcup^{+\infty}_{-\infty}E^s_{x_n}\oplus E^u_{x_n}$ that is preserved by the 
derivative $Df$ and satisfies the conditions:
\begin{enumerate}
 \item $||Df^n(v)||\leq C \lambda^n||v||$, for $v\in E^s, n\geq 0$
\item $||Df^n(v)||\geq C^{-1} \lambda^{-n}||v||$, for $v\in E^u, n\geq 0$.
\end{enumerate}
\end{Def}
\begin{Rem}
It is noticeable that we do not ask for a splitting of the whole tangent bundle $TM=E^s\oplus E^u$.  It may happen that $E^u_{(x_n)_{n\in\mathbb{Z}}}\neq E^u_{(y_n)_{n\in\mathbb{Z}}}$ though $x_0=y_0$. There is a construction of a mapping that is close to an algebraic Anosov mapping while at a point it has many different local unstable manifolds in \cite{Przytycki}. This construction can also be reckoned as an explanation of the non stability of Anosov mappings. 
 For $E^s$, $E^s_{(x_n)_{n\in\mathbb{Z}}}$ only depends on $x_0$. Of course,
there are special systems for which $E^u_{x}$ does not depend on the orbits containing $x$. A classical example of such  mapping
is any algebraic mapping of the torus, such as $\left[\begin{matrix} n & 1 \\ 1 & 1 \\\end{matrix}\right]$ for $n\geq 2.$
\end{Rem}
The following theorem is the classical Oseledec's theorem, a version of the Multiplicative Ergodic Theorem for differentiable 
mappings.
 \begin{Thm}\label{Ose}Let $f$ be a $\mathcal{C}^1$ mapping on $M$. Then there exists a Borel subset $G\subset M$ with $f(G)\subset G$ 
and $\mu(G)=1$ for any $\mu\in \mathcal{M}_{inv}(M)$, such that the following properties hold.
\begin{enumerate}\item There is a measurable integer function $r:G\rightarrow \mathbb{Z}^+$ with $r\circ f=r.$
\item For any $x\in G$, there are real numbers
$$+\infty>\lambda_1(x)>\lambda_2(x)>\cdots>\lambda_{r(x)}(x)\geq-\infty,$$
where $\lambda_{r(x)}(x)$ could be $-\infty.$
\item If $x\in G,$ there are linear subspaces
$$V^0(x)=T_xM\supset V^1(x)\supset\cdots\supset V^{r(x)}(x)={0}$$
of $T_xM.$
\item If $x\in G$ and $1\leq i\leq r(x),$ then 
$$\lim_{n\rightarrow \infty}\frac{1}{n}\log |D_xf^n\xi|=\lambda_i(x)$$
for all $\xi\in V^{i-1}(x)\backslash V^i(x).$ Moreover,
$$\lim_{n\rightarrow \infty}\frac{1}{n}\log|\det (D_xf^n)|=\sum_{i=1}^{r(x)}\lambda_i(x)m_i(x),$$
where $m_i(x)=\dim V^{i-1}(x)-\dim V^i(x)$ for all $1\leq i\leq r(x)$.
\item $\lambda_i(x)$ is measurably defined on $\{x\in G|r(x)\geq i\}$ and $f$-invariant, i.e. $\lambda_i(fx)=\lambda_i(x).$
\item $D_xf(V^i(x))\subset V^i(f(x))$ if $i\geq 0$. 
\end{enumerate}
The numbers $\{\lambda_i(x)\}_{i=1}^{r(x)}$ defined above are called the Lyapunov exponents of $f$ at point $x$ and $m_i(x)$ is called the multiplicity of $\lambda_i(x)$.
\end{Thm}
\begin{Rem}By  Oseledec's ergodic theorem \ref{Ose} for maps, we only have a filtration type splitting in the tangent space which kills lots of skills in Pesin theory.
Thus,
there only exist well defined stable manifolds for mappings. Nevertheless, it is comforting that by Pugh and Shub's theorem \ref{Oseledec} below, 
we can find  full measure orbits such that along these orbits, there exist well defined invariant unstable manifolds.  It is worth to note that what underlies this fact is the multiplicative ergodic theorem for non-invertible maps given by  Ruelle or Ma\~{n}e's \cite{Rue, Mane}.  
\end{Rem}

It is a common idea to consider the inverse limit space for mappings.  Let $\tilde{f}: M^f\rightarrow M^f$ be the induced map where $\tilde{f}$ is the shift
map.  Let $\pi:M^f\rightarrow M$ be naturally defined by $\pi((x_n)_{n\in\mathbb{Z}})=x_0$, then $\pi\circ \tilde{f}=f\circ\pi$. 
The tangent bundle of $M$, $TM$ pulls back to $TM^f$ on $M^f$ and $Df$ extends to $D\tilde{f}$, i.e.
$D\tilde{f}$ is a continuous bundle mapping covering the homeomorphism $\tilde{f}$ of the compact base $M^f$. $D\tilde{f}$ is a linear map on each fiber.  
\begin{Def}\label{Coc}Let $f$ be a $\mathcal{C}^{1+\alpha}$ mapping on compact manifold $M$ and $\tilde{f}$, $M^f$
defined as above. Let $L(d,\mathbb{R})$ denote the group of  $d\times d$ matrices over $\mathbb{R}^d$.  For any measurable function
$A:M^f\rightarrow L(d,\mathbb{R})$, let $\mathcal{A}:M^f\times\mathbb{Z}\rightarrow L(d,\mathbb{R})$ defined by
$$\mathcal{A}(\tilde{x},m)=A(\tilde{f}^m(\tilde{x}))\cdots A(\tilde{x})\,\, for\,\, m\geq0,$$
$$\mathcal{A}(\tilde{x},m)=A(\tilde{f}^m\tilde{x})^{-1}\cdots A(\tilde{f}^{-1}\tilde{x})^{-1}\,\, for\,\, m<0. $$
Then it follows  
\begin{equation}\label{cocycles}\mathcal{A}(\tilde{x},m+k)=\mathcal{A}(\tilde{f}^k(\tilde{x}),m)\mathcal{A}(\tilde{x},k).
\end{equation}
We call  $\mathcal{A}:M^f\times \mathbb{Z}\rightarrow L(d,\mathcal{R})$ 
 a {\bf measurable linear cocycle} over $f$, or simply a cocycle.
\end{Def}
Thus, there is a natural measurable cocycle over $f$ associated with  $D\tilde{f}$. We abuse notation $D\tilde{f}: M^f\times \mathbb{Z}\rightarrow L(d,\mathcal{R})$ defined as following.  
 \begin{Def}The measurable  cocycle $D\tilde{f}$  over $f$ is defined as following 
$$
D\tilde{f}^m(\tilde{x})=
\begin{cases}
D_{x_0}f^m=D_{x_0}f\circ\cdots\circ D_{x_m}f, & \text{if }m>0; \\
Id, & \text{if }m=0;\\
(D_{x_m}f^{m})=(Df_{x_m})^{-1}\circ\cdots\circ(Df_{x_{-1}})^{-1} & \text{if }m<0.
\end{cases}
$$
\end{Def}
\begin{Rem}We should notice that inverse limit space  isn't a manifold. It is just a topological space with linear cocycle $\tilde{D}f$ over it  and  the dimension
of $M^f$ is even infinite usually. Although we can not say $\tilde{D}f$ is the derivative to $\tilde{f}$, it is a linear cocycle over $\tilde{f}$.  
\end{Rem}

Invariant measures in $M^f$ can be projected down to invariant measures in $M$ by projection $\pi$. The following lemma 
says $\mathcal{M}_{inv}{M^f}$ is equivalent to $\mathcal{M}_{inv}{M}$.
\begin{Lem}\cite{MXZ}\label{EM}Let $T$ be a continuous map on $M$. For any $T$ -invariant Borel probability
 measure $\mu$ on $M$, there exists a unique $\tilde{T}$-invariant Borel probability measure $\tilde{\mu}$ on
$M^T$ such that $\pi\tilde{\mu}=\mu$. Moreover, $h_{\tilde{\mu}}(\tilde{T})=h_{\mu}T$. 
\end{Lem}
\begin{Prop}\cite{Shub1}\label{Oseledec}For any invariant measure $\mu$ of $f:M\rightarrow M$, there exists a Borel set $\tilde{\Lambda}\subset M^f$, such that
\begin{enumerate} 
\item$\tilde{f} \tilde{\Lambda}
=\tilde{\Lambda}$,
\item $\tilde{\mu}(\tilde{\Lambda})=1$
 \item for every $\tilde{x}=\{x_n\}_{n\in \mathbb{Z}
}\in\tilde{\Lambda}$, there are  splittings of the tangent space $T_{x_n}M$,
$$T_{x_n}M=E_1(x_n)\oplus E_2(x_n)\oplus\cdots\oplus E_{\tilde{r}(\tilde{x})}(x_n)\oplus F_{\infty}(x_n)$$
and numbers $\infty>\tilde{\lambda}_1(\tilde{x})\geq\tilde{\lambda}_s(\tilde{x})\geq\cdots\geq\tilde{ \lambda}_{\tilde{r}(\tilde{x})}
(\tilde{x})>-\infty$ and $\tilde{m}_i(\tilde{x})$, satisfying the following properties:
\begin{enumerate}
\item $D_{x_n}f|E_i(x_n)$ is an isomorphism, $\forall n\in \mathbb{Z}$.
\item $\tilde{r}(\cdot),\tilde{\lambda}(\cdot)$ and $\tilde{m}(\cdot)$ are $\tilde{f}$ measurable and invariant, i.e.
$$\tilde{r}(\tilde{f}(\tilde{x}))=\tilde{r}(\tilde{x}),\,\,\tilde{\lambda}_i(\tilde{f}(\tilde{x}))=\tilde{\lambda}_i(\tilde{x})\,\, and\,\, \tilde{m}_i(\tilde{f}(\tilde{x}))$$
for each $i=1,\cdots,\tilde{r}(\tilde{x}).$
\item $\text{dim}E_i(\tilde{x})=\tilde{m}_i(\tilde{x})$ for all $n\in\mathbb{Z}$ and $1\leq i\leq \tilde{r}(\tilde{x}).$
\item $$\lim_{n\rightarrow +\infty}\frac{1}{n}\log |\tilde{D}f^n|F_{\infty}(x_n)|=-\infty.$$
\item $$\lim_{n\rightarrow \pm\infty}\frac{1}{n}\log |\tilde{D}f^n(v)|=\tilde{\lambda}_i(\tilde{x}),$$ for all $0\neq v\in E_{i}(x_n), 1\leq i\leq \tilde{r}(\tilde{x}).$
\item If $0\neq v_n\in F_{\infty}(x_n)$ and there are $v_m\in F_{\infty}(x_m)$ for $m<n$ such that $Df^{n-m}(v_m)=v_n$ then $\lim_{n\rightarrow \infty}\log |v_j|=\infty.$
\item $F_{\infty}(x_n)=K(x_n)\oplus G_{\infty}(x_n)$ where $D_{x_n}f^{i}|_{K(x_n)}$ is identically 0 for some $i$ and $Df|_{G_{\infty}(x_i)}$ is an
isomorphism.
\item The splitting is  measurable with respect to $\tilde{x}$  and the angles between any two associated subspaces vary sub-exponentially
under iteration, i.e. 
$$\lim_{n\rightarrow \pm \infty}\frac{1}{n}\angle(E_i(x_n),E_j(x_n))=0,1\leq i,j\leq\tilde{r}(\tilde{x}) \text{ and}$$ $$\lim_{n\rightarrow \pm \infty}\frac{1}{n}\angle(E_i(x_n),F_{\infty}(x_n))=0, 1\leq i\leq \tilde{r}(\tilde{x}).$$
\end{enumerate}
\end{enumerate}
\end{Prop}

Although we do not use it in this paper, we state a result in \cite{Shub1} about the existence of unstable manifolds along orbits.  Let $f:M\rightarrow M$ be a $\mathcal{C}^{1+\beta}$ mapping and let $\mu$ be an invariant measure 
for $f$ with no zero exponent, then for almost all full orbits of $f$ there are stable and unstable disc families which are Borel, vary 
sub-exponentially along orbits and are invariant.
\begin{Rem}It is worth to note that the  understanding of dynamics in inverse limit spaces is far away from the understanding of the original maps. For example, it is well known that a non-invertible mapping on a compact manifold is in general not stable except it is expanding\cite{Przytycki}. Even so, for  Anosov mappings, the dynamical structure of its orbit space is stable under $\mathcal{C}^1$ small perturbations\cite{Liu1}\cite{Mane2}.  
\end{Rem}

 \subsection{Shadowing lemma for sequences of mappings}
 A lot of shadowing problems can be reduced to the following ``abstract" shadowing problem. 
Let $H_k$ be a sequence of Banach spaces ($k\in\mathbb{Z}$ or $k\in\mathbb{Z^{+}}$), we denote by $|\cdot|$ norms in $H_k$ and by $||\cdot||$ the corresponding operator norms for linear operators. Let us emphasize that the spaces $H_k$ are not assumed to be isomorphic. 

Consider a sequence of mappings 
$$\phi_k:H_k\rightarrow H_{k+1}$$ of the form 
$$\phi_k(v)=A_kv+w_{k+1}v$$ where $A_k$ are linear mappings.

It is assumed that the values $|\phi_k(0)|$ are uniformly small, say, $|\phi_k(0)|\leq d.$ We are looking for a sequence $v_k\in H_k$ such that $\phi_k(v_k)=v_{k+1}$ and the values $|v_k| $ are uniformly small, for example, the inequalities $$\sup_{k}|v_k|\leq Ld$$ hold with a constant $L$ independent of $d$. 
\begin{Thm}\cite{Pily}
Assume that 
\begin{enumerate}
\item there exist numbers $\lambda\in(0,1), N\geq 1,$ and projectors $P_k,Q_k:H_k\rightarrow H_k$ such that 
\begin{enumerate}
\item $||P_k||,||Q_k||\leq N, P_k+Q_k=I;$
\item $||A_k|_{S_k}||\leq \lambda, A_kS_k\subset S_{k+1};$
\end{enumerate}
\item if $U_{k+1}\neq\{0\}$, then there exist linear mappings $B_k:U_{k+1}\rightarrow U_k$ such that $$B_kU_{k+1}\subset U_k, ||B_k||\leq \lambda, A_kB_k|_{U_{k+1}}=I;$$
\item there exist numbers $k,\Delta>0$ such that inequalities 
$$|w_{k+1}(v)-w_{k+1}(v')|\leq k|v-v'|\text{ for } |v|,|v'|\leq \Delta$$ and $kN_1<1$ hold where $N_1=N\frac{1+\lambda}{1-\lambda}.$ 
\end{enumerate}
Set $L=\frac{N_1}{1-kN_1}, d_0=\frac{\Delta}{L}.$ If for a sequence of mapping $\phi_k$, we have $|\phi_k(0)|\leq d \leq d_0$, then there exist points $v_k\in H_k$ such that $\phi_k(v_k)=v_{k+1}$ and $|v_k|\leq Ld$. 
\end{Thm} 
We give sufficient conditions for the uniqueness of  sequence $\{v_k\}$. 
\begin{Thm}\cite{Pily} \label{Shadowing}Assume that 
\begin{enumerate}
\item there exist numbers $\lambda\in(0,1), N\geq 1,$ and projectors $P_k,Q_k:H_k\rightarrow H_k$ such that 
\begin{enumerate}
\item $||P_k||,||Q_k||\leq N, P_k+Q_k=I;$
\item $||A_k|_{S_k}||\leq \lambda, A_kS_k\subset S_{k+1};$
\end{enumerate}
\item $A_kU_k\subset U_{k+1}$ and $||A_k|_{U_k}||\geq\frac{1}{\lambda}$;
\item there exist numbers $k_0,\Delta>0$ such that$$|w_{k+1}(v)-w_{k+1}(v')|\leq k|v-v'|\text{ for } |v|,|v'|\leq \Delta$$ and for $k=Nk_0$ the inequalities 
$$\lambda+2k<1, \frac{1}{\lambda}-2k\geq \gamma>1, \frac{\lambda}{\gamma}+\frac{2k}{\gamma}<1$$ 
are fulfilled.
\end{enumerate}
Then the relations $$\phi_k(v_k)=v_{k+1},\phi_k(u_k)=u_{k+1}, |v_k|,|u_k|\leq \Delta,k\in\mathbb{Z},$$
imply that $v_k=u_k,k\in\mathbb{Z}.$
\end{Thm}

\section{Construction of Horseshoe}
In this section, we focus our attention on the set $\tilde{\Lambda}$ which is given in Proposition \ref{Oseledec}.
\begin{Def}Let $\mu$ be an ergodic hyperbolic probability measure for a $\mathcal{C}^1$ mapping $f: M\rightarrow M$. Correspondingly,
we have  inverse limit space $M^f$, shift map $\tilde{f}$ and ergodic measure $\tilde{\mu}$ with respect to $\mu$. 
A compact positive measure set $\tilde{\Lambda}(\eta)
\subset \tilde{\Lambda}$ is called a {\bf $\eta$-uniformity block } for $\mu$(with tolerance $\eta>0$)
if there exists $K>0$ and a measurable map $C_{\eta}:\tilde{\Lambda}\rightarrow GL(d,\mathbb{R})$ which is continuous on subset $\tilde{\Lambda}(\eta)$
such that:
\begin{enumerate}
\item$\max\{||C^{-1}_{\eta}(\tilde{f}^n(\tilde{x}))||,||C_{\eta}(\tilde{f}^n(\tilde{x}))||\}< K\exp (\eta |n|)$, for each $\tilde{x}\in\tilde{\Lambda}$ and $n\in\mathbb{Z}.$
\item Denote $\tilde{\lambda}^{+}= \min\{\tilde{\lambda}_i>0\}$, $\tilde{\lambda}^{-}= \max\{\tilde{\lambda}_i<0\}$ and $\tilde{\lambda}=\min\{\tilde{\lambda}^+,-\tilde{\lambda}^{-}\}$ where $\tilde{\lambda}_1>\tilde{\lambda}_2>\cdots>\tilde{\lambda}_k>0>\tilde{\lambda}_{k+1}>\cdots>\tilde{\lambda}_{s}$ are distinct Lyapunov exponents of $\mu$, with multiplicities
$n_1,\cdots,n_s\geq1$, then there  exists $A_{1}=A_{1}(\tilde{x})\in GL(\sum_{1}^{k}n_i,\mathbb{R})$ and $A_{2}=A_{2}(\tilde{x})\in L(\sum_{k+1}^{s}n_i,\mathbb{R})$ such that 
$$||A_{1}(\tilde{x})^{-1}||^{-1}\geq e^{\tilde{\lambda}-\eta}, \,\,\, ||A_{2}(\tilde{x})||\leq e^{-\tilde{\lambda}+\eta}$$
and $C_{\eta}(\tilde{f}(\tilde{x}))\cdot Df(x_0)\cdot C_{\eta}^{-1}(\tilde{x})= diag(A_{1}(\tilde{x}), A_{2}(\tilde{x}))$.
\end{enumerate}
\end{Def}
\begin{Rem}$C_{\eta}(\tilde{x})$ here is a transformation of coordinates such that  splitting $E^u(\tilde{x})\oplus E^s(\tilde{x})$ are mapped to $e_1\oplus e_2$ where $e_1=(1^{d_u},0^{d_s}), e_2=(0^{d_u},1^{d_s})$ and $d_s,d_u$ means the dimension of stable bundle and unstable bundle respectively. Under this new coordinates, derivatives $Df$ present enough hyperbolicity at the first iteration. So the norm of $C_{\eta}(\tilde{x})$ is determined by the angle of $E^s(\tilde{x})$ and $E^u(\tilde{x})$ and large $N$ such that $||Df^N(v)||$ shows enough hyperbolicity which is determined by $K,\eta$ and $||Df||$.  
\end{Rem}
\begin{Def}A $\mu-$measurable map $C:M^f\rightarrow GL(n,\mathbb{R})$ is said to be {\bf tempered } with respect to $f$, or simply tempered, if for $\mu-$almost every $\tilde{x}\in M^f$
$$\lim_{n\rightarrow +\infty}\frac{1}{n}\log ||C^{\pm}(\tilde{f}^n(\tilde{x}))||=0.$$
\end{Def}
The following lemma is a technical but crucial lemma in smooth ergodic theory.
\begin{Lem}\label{Tempering} \cite{KatMe}  (Tempering-Kernel Lemma) Let $f:X\rightarrow X$ be a measurable transformation. If $K:X\rightarrow\mathbb{R}$ is a positive measurable tempered function, then for any $\epsilon>0$, there exists a positive measurable function $K_{\epsilon}:X\rightarrow\mathbb{R}$ such that $K(x)\leq K_{\epsilon}(x)$ and $$e^{-\epsilon}\leq \frac{K_{\epsilon}(f(x))}{K_{\epsilon}(x)}\leq e^{\epsilon}.$$
\end{Lem}

\begin{Thm}\label{OPReduction}{\bf(Oseledec-Pesin $\epsilon$ Reduction theorem for mappings)}
Suppose that $D\tilde{f}:\tilde{\Lambda}\rightarrow L(d,\mathbb{R})$ is the measurable cocycle over shift $\tilde{f}:M^f\rightarrow M^f$ where $\tilde{\Lambda}$ is the full measure set given in Proposition \ref{Oseledec}. Then there exists a measurable $\tilde{f}$- invariant function $r:\tilde{\Lambda} \rightarrow \mathbb{N} $ and number $\lambda_1(\tilde{x}),\cdots,\lambda_{r(\tilde{x})}(\tilde{x})\in\mathbb{R}$ and $\l_1(\tilde{x}),\cdots,\l_{r(\tilde{x})}(\tilde{x})\in\mathbb{N}$ depending only on $\tilde{x}$ with $\sum l_i(\tilde{x})=d$ such that for every $\epsilon>0$ there exists a tempered map $$C_{\epsilon}:\tilde{\Lambda}\rightarrow GL(d,\mathbb{R})$$ 
such that for almost every $\tilde{x}=\{x_i\}_{i\in\mathbb{Z}}\in \tilde{\Lambda}$ the cocycle $A_{\epsilon}(\tilde{x})=C_{\epsilon}(\tilde{f}(\tilde{x}))Df(x_{0})C^{-1}_{\epsilon}(\tilde{x})$ has the following Lyapunov block form
\[A_{\epsilon}(\tilde{x})=\begin{pmatrix}
\ A^1_{\epsilon}(x_{0})\\
&\ A^2_{\epsilon}(x_{0})
\end{pmatrix},\]
where  $A^{1}_{\epsilon}(\tilde{x})$ is a $\sum_{j=1}^kl_j(\tilde{x})\times \sum_{j=1}^kl_j(\tilde{x})$,  $A^{2}_{\epsilon}(\tilde{x})$ is a $\sum_{j=k+1}^r l_j(\tilde{x})\times\sum_{j=k+1}^r l_j(\tilde{x})$ matrix and 
\[||(A^{1}_{\epsilon}(\tilde{x}))^{-1}||^{-1}\geq e^{\lambda(\tilde{x})-\epsilon} ,  ||A^{2}_{\epsilon}(\tilde{x})||\leq e^{-\lambda(\tilde{x})+\epsilon}.\]
\end{Thm}
\begin{proof}This result follows directly from the same proof as the diffeomorphism case (See Theorem S.2.10\cite{KatMe}).  We give a sketch here.
If $\tilde{x}\in \tilde{\Lambda}$ then $T_{x_0}M=E_1(x_0)\oplus E_2(x_0)$ where $E_1(x_0)$ is the expanding part and $E_2(x_0)$ is the contracting part.  Define a new scalar product on each $E_1(x_0)$ and $\epsilon>0$ as follows:
If $u,v\in E_1(x_0)$ then
$$<u,v>'_{\tilde{x},1}:=\sum_{m=0}^{+\infty}<D\tilde{f}|_{E^1}^{-m}(\tilde{x})u, D\tilde{f}|_{E^1}^{-m}(\tilde{x})v)>e^{2m\lambda}e^{-2\epsilon m},$$
where $<\cdot,\cdot>$ denotes the standard scalar product on $\mathbb{R}^n;$
If $u,v\in E_2(x_0)$ then
$$<u,v>'_{\tilde{x},2}:=\sum_{m=0}^{+\infty}<D\tilde{f}|_{E^2}^{m}(\tilde{x})u, D\tilde{f}|_{E^2}^{m}(\tilde{x})v)>e^{-2m\lambda}e^{-2\epsilon m},$$
where $<\cdot,\cdot>$ denotes the standard scalar product on $\mathbb{R}^n.$
Now according to  the definition of Lyapunov exponents $\tilde{\lambda}_i(\tilde{x}),$ for each $\tilde{x}\in\tilde{\Lambda}$  and $\epsilon>0$ there exists a constant $C(\tilde{x},\epsilon) $ such that 
$$||D\tilde{f}^{-m}m(\tilde{x})v||\leq C(\tilde{x},\epsilon)\exp^{m\lambda(\tilde{x})}e^{\epsilon m/2}||v||,\forall m\in \mathbb{N}, u,v\in E^1;$$
$$||D\tilde{f}^m(\tilde{x})v||\leq C(\tilde{x},\epsilon)\exp^{-m\lambda(\tilde{x})}e^{\epsilon m/2}||v||,\forall m\in \mathbb{N}, u,v\in E^2,$$
 therefore $<u,v>'_{\tilde{x},i}\leq C^2(\tilde{x},\epsilon)\sum_{m\in\mathbb{N}}e^{-m\epsilon},$
which implies that $<u,v>'_{\tilde{x},\epsilon}$ is well defined.

We recall that $D\tilde{f}^{m+1}(\tilde{x})=D\tilde{f}^{m}(\tilde{f}\tilde{x})Df(x_0)$, whence
\begin{eqnarray*}
<Df(x_0)v,Df(x_0)v>'_{\tilde{f}\tilde{x},1}&=&\sum_{m\in\mathbb{N}}||D\tilde{f}^{-m}(\tilde{f}(\tilde{x}))(Df(x_0)v)||^2e^{2\lambda(\tilde{x})m}e^{-2\epsilon|m|}\\
&=&\sum_{m=0}^{+\infty}||D\tilde{f}^{-m+1}(\tilde{x})v||^2e^{2\lambda(\tilde{x})m}e^{-2\epsilon m}\\
&\geq& e^{2(\lambda(\tilde{x})-\epsilon)}\sum^{+\infty}_{m=0}||D\tilde{f}^{-m}(\tilde{x})v||^2e^{2\lambda(\tilde{x})m}e^{-2\epsilon m}\\
&\geq&e^{2(\lambda(\tilde{x})-\epsilon)}<u,v>'_{\tilde{x},1},
\end{eqnarray*}
where $u,v\in E^1$
So we have 
\begin{equation}\label{norm}
\frac{<Df(x_0)v,Df(x_0)v>'_{\tilde{f}\tilde{x},1}}{<v,v>'_{\tilde{x},1}}\geq e^{2(\lambda(\tilde{x})-\epsilon)},\forall u,v\in E^1
\end{equation}
Similarly, we have 
\begin{eqnarray*}
<Df(x_0)v,Df(x_0)v>'_{\tilde{f}\tilde{x},2}&=&\sum_{m\in\mathbb{N}}||D\tilde{f}^{m}(\tilde{f}(\tilde{x}))(Df(x_0)v)||^2e^{-2\lambda(\tilde{x})m}e^{-2\epsilon|m|}\\
&=&\sum_{m=0}^{+\infty}||D\tilde{f}^{-m+1}(\tilde{x})v||^2e^{2\lambda(\tilde{x})m}e^{-2\epsilon m}\\
&\leq& e^{2(-\lambda(\tilde{x})+\epsilon)}\sum^{+\infty}_{m=0}||D\tilde{f}^{m}(\tilde{x})v||^2e^{2\lambda(\tilde{x})m}e^{-2\epsilon m}\\
&\leq&e^{2(-\lambda(\tilde{x})+\epsilon)}<u,v>'_{\tilde{x},2},
\end{eqnarray*}
where $u,v\in E^2$
So we have 
\begin{equation}\label{norm2}
\frac{<Df(x_0)v,Df(x_0)v>'_{\tilde{f}\tilde{x},2}}{<v,v>'_{\tilde{x},2}}\leq e^{2(-\lambda(\tilde{x})+\epsilon)},\forall u,v\in E^2
\end{equation}
To extend the scalar product to $T_{x_0}M$, consider $$<u,v>'_{\tilde{x}}=\sum^{2}_{i=1}<u_i,v_i>'_{\tilde{x},i},$$
where $v_i$ is the projection of $v$ to $E^i(\tilde{x})$. Now let $C_{\epsilon}(\tilde{x})$ be the positive symmetric matrix such that if $u,v\in T_{x_0}M$ then 
$$<u,v>'_{\tilde{x}}=<C_{\epsilon}(\tilde{x})u,C_{\epsilon}(\tilde{x})v>$$ and define $$A_{\epsilon}(\tilde{x})=C_{\epsilon}(\tilde{f}(\tilde{x}))Df(x_0)C^{-1}_{\epsilon}(\tilde{x}).$$
Thus if $u,v\in E_i(\tilde{x}),$ then 
\begin{eqnarray*}
<Df(x_0)u,Df(x_0)v>'_{\tilde{f}\tilde{x},i}&=&<C_{\epsilon}(\tilde{f}\tilde{x})Df(x_0)u,C_{\epsilon}(\tilde{f}\tilde{x})Df(x_0)v>\\
&=&<A_{\epsilon}(\tilde{x})C_{\epsilon}(\tilde{x})u,A_{\epsilon}(\tilde{x})C_{\epsilon}(\tilde{x})v>
\end{eqnarray*}
Applying $v=C_{\epsilon}(\tilde{x})u$ in inequality (\ref{norm},\ref{norm2}), we get
\[||(A^{1}_{\epsilon}(\tilde{x}))^{-1}||^{-1}\geq e^{\lambda(\tilde{x})-\epsilon} ,  ||A^{2}_{\epsilon}(\tilde{x})||\leq e^{-\lambda(\tilde{x})+\epsilon}.\]

It remains to prove that $C_{\epsilon}(x)$ is tempered.  Since the angles between the different subspaces satisfy a sub exponential lower estimate due to Theorem\ref{Oseledec}, it is enough to consider just block matrices. Set $B_N:=\{\tilde{x}\in\tilde{\Lambda}|||C_{\epsilon}^{\pm}(\tilde{x})||<N\}.$ For some $N>0$ large enough, by the Poincare Recurrence Theorem, there exists a set $Y\subset B_N$ such that $ \mu(B_N\backslash Y)=0 $ and the orbit of $\tilde{y}\in Y$ returns infinitely many times to $Y$. Thus
let $m_k$ be a sequence such that $\tilde{f}^{m_k}(\tilde{x})\in Y$ for all $k$. Then 
$$A_{\epsilon}(\tilde{y},m_k):=A_{\epsilon}(\tilde{f}^{m_k}\tilde{y})\cdots A_{\epsilon}(\tilde{y})\leq ||C_{\epsilon}(\tilde{f}^{m_k}\tilde{y})||\,||D\tilde{f}^{m_k}\tilde{y}||\,||C^{-1}_{\epsilon}(\tilde{y})||$$
and therefore for almost every point $\tilde{y}\in Y$ the spectra of $A_{\epsilon}$ and $D\tilde{f}$ are the same.
 Since $N$ is chosen arbitrarily this is true for almost every $\tilde{x}\in\tilde{\Lambda}.$ Observe that 
 $$C^{-1}_{\epsilon}(\tilde{f}^n\tilde{x})=D\tilde{f}^m(\tilde{x})C^{-1}_{\epsilon}(\tilde{x})(A_{\epsilon}(\tilde{x},n))^{-1}\text{ and } D\tilde{f}^n(\tilde{x})=C_{\epsilon}^{-1}(\tilde{f}^n\tilde{x})A_{\epsilon}(\tilde{x},n)C_{\epsilon}(\tilde{x}),$$
 so by taking the growth rates in both equations we find that $$\lim_{n\rightarrow\infty}\frac{1}{n}\log ||C^{\pm}_{\epsilon}(\tilde{f}^n\tilde{x})||=0$$
 for all $\tilde{x}\in\tilde{\Lambda}$ for which $A_{\epsilon}$ and $D\tilde{f}$ have the same spectrum. 
\end{proof}
\begin{Thm} If $f:M\rightarrow M$ is $\mathcal{C}^1$ mapping and $\mu$ is an ergodic hyperbolic probability measure for $f$, then for every $\eta >0$
there exists a uniformity block $\tilde{\Lambda}(\eta)$  of tolerance $\eta$ for $\mu$. Moreover, $\tilde{\Lambda}(\eta)$ can be chosen to have measure 
arbitrarily close to 1 with suitable choose of $K$.
 \end{Thm}
\begin{proof}Applying Oseledec and Pesin's Reduction theorem given in Theorem  \ref{OPReduction} and Lusin theorem, we can get this theorem.
\end{proof}

The assumption of the existence of hyperbolic ergodic measure ensures hyperbolicity.  The norm we used in the proof of Theorem \ref{OPReduction} is called Lyapunov norm. In non-uniformly hyperbolic case, Lyapunov metric is a powerful technique.
 Under Lyapunov norm, one can get 
uniformly hyperbolic property.
From the definition of uniformity blocks, we can see that for any $\tilde{x}\in\tilde{\Lambda}(\eta)$  there exist
linearization and  diagonalizaton of $f$ along  orbit $\tilde{x}$. In the following theorem, we want to estimate the $\mathcal{C}^1$ distance between  diagonalization and our maps under local charts. 
Next theorem says for uniformity block there exists uniformly bound which depends only on the tolerance of the block.
We need a promotion from points to their neighborhoods, which requires higher regularity  $\mathcal{C}^{1+\alpha}.$
\begin{Thm}\label{Pesin} {\bf (Pesin's argument for mappings)} Let $f: M\rightarrow M$ be a $\mathcal{C}^{1+\alpha}$ mapping preserving the ergodic hyperbolic 
probability measure
$\mu$. Let $\tilde{\Lambda}(\eta)$ be a uniformity block of tolerance $\eta$ for $\mu$ . Then there exist $K>0$, $\xi_0>0$, a measurable map $C_{\eta}:\tilde{\Lambda}\rightarrow GL(d,\mathbb{R})$ which is continuous on $\tilde{\Lambda}(\eta)$
$$C_{\eta}(\tilde{x}):T_{x_0}M\rightarrow \mathbb{R}^d; \,\, \tilde{x}=\{x_n\}_{n\in \mathbb{Z}}\in \tilde{\Lambda}$$
and a measurable function $\xi:\tilde{\Lambda}\rightarrow \mathbb{R}^{+}$ which is also continuous on $\tilde{\Lambda}(\eta)$
such that:
\begin{enumerate}
\item$\max\{||C^{-1}_{\eta}(\tilde{f}^n(\tilde{x}))||,||C_{\eta}(\tilde{f}^n(\tilde{x}))||\}< K\exp (\eta |n|)$, for each $\tilde{x}\in\tilde{\Lambda}$ and $n\in\mathbb{Z}.$
\item Denote $\tilde{\lambda}^{+}= \min\{\tilde{\lambda}_i>0\}$, $\tilde{\lambda}^{-}= \max\{\tilde{\lambda}_i<0\}$ and $\tilde{\lambda}=\min\{\tilde{\lambda}^+,-\tilde{\lambda}^{-}\}$ where $\tilde{\lambda}_1>\tilde{\lambda}_2>\cdots>\tilde{\lambda}_k>0>\tilde{\lambda}_{k+1}>\cdots>\tilde{\lambda}_{s}$ are distinct Lyapunov exponents of $\mu$, with multiplicities
$n_1,\cdots,n_s\geq1$, then there  exists $A_{1}=A_{1}(\tilde{x})\in GL(\sum_{1}^{k}n_i,\mathbb{R})$ and $A_{2}=A_{2}(\tilde{x})\in L(\sum_{k+1}^{s}n_i,\mathbb{R})$ such that 
$$||A_{1}(\tilde{x})^{-1}||^{-1}\geq e^{\tilde{\lambda}-\eta}, \,\,\, ||A_{2}(\tilde{x})||\leq e^{-\tilde{\lambda}+\eta}$$
and $C_{\eta}(\tilde{f}(\tilde{x}))\cdot Df(x_0)\cdot C_{\eta}^{-1}(\tilde{x})= diag(A_{1}(\tilde{x}), A_{2}(\tilde{x}))$.
\item $\xi_0 e^{-\eta n}\leq \xi(\tilde{f}^n(\tilde{x}))\leq \xi_0e^{\eta n}$, for all $\tilde{x}\in \tilde{\Lambda}(\eta)$ and $n\in\mathbb{Z}$;

\item $f(B(x_n, K^{-1}\xi(\tilde{f}^n(\tilde{x})))\subset B(x_{n+1}, \xi(\tilde{f}^{n+1}(\tilde{x})))$;

\item there are $\mathcal{C}^2$ charts $\phi_{\tilde{x}}: B(x_0, K^{-1}\xi(\tilde{x}))\rightarrow \mathbb{R}^d $ for any $\tilde{x}\in\tilde{\Lambda}$ such that 
  \begin{enumerate}
  \item $\tilde{x}\longmapsto \phi_{\tilde{x}}$ is continuous in the $\mathcal{C}^1$ topology on the target $\tilde{\Lambda}(\eta).$
   \item $\phi_{\tilde{x}}(x_0)=0,$ and $D\phi_{\tilde{x}}(x_0)=C_{\eta}(\tilde{x});$	
  \item	on the set $\phi_{\tilde{x}}(B(x_0,K^{-1}\xi(\tilde{x}))),$ we have 

$d_{\mathcal{C}^1}(\phi_{\tilde{f}\tilde{x}}\circ f\circ \phi^{-1}_{\tilde{x}}, C_{\eta}(\tilde{f}(\tilde{x}))\cdot Df(x_0)\cdot C^{-1}_{\eta}(\tilde{x}))<\eta;$
  \item for all $\tilde{x}\in \tilde{\Lambda}(\eta)$ and for all $y,y'\in B(x_0,K^{-1}\xi(\tilde{x})),$ we have 
$K^{-1}d(y,y')\leq||\phi_{\tilde{x}}(y)-\phi_{\tilde{x}}(y')||\leq Ke^{\eta \tau(\tilde{x})}d(y,y')$ where $\tau(\tilde{x}):=\{\text{ the first return time of } \tilde{x}\text{ to }\tilde{\Lambda}(\eta)\}$. 
  \end{enumerate}
\end{enumerate}
\end{Thm}
\begin{proof}Our aim is to construct for almost every $\tilde{x}\in\tilde{\Lambda}=\{x_n\}_{n\in\mathbb{Z}}$, a neighborhood $N(x_n)$ such that $f$ acts on $N(x_n)$ very much like the linear map $A_{\epsilon}(x_n)$ in a neighborhood of the origin.  This  follows from the techniques used in Pesin's original proof in Theorem S.3.1 \cite{KatMe}.  

For $\eta>0$ and let $\tilde{\Lambda}\subset M^f$ be the set for which the Oseledec-Pesin $\eta$-Reduction Theorem\ref{OPReduction} works. For each $\tilde{x}\in\tilde{\Lambda}$ consider $(C_{\eta}(\tilde{f}^n(\tilde{x})))_{n\in\mathbb{Z}}$, linear maps from $T_{x_n}M$ to $\mathbb{R}^d,$  where $C_{\eta}$ is the Lyapunov transformation of coordinates given in Theorem \ref{OPReduction}  such that  $A_{\eta}(\tilde{x})=C_{\eta}(\tilde{f}(\tilde{x}))Df(x_{0})C^{-1}_{\eta}(\tilde{x})$ has the following Lyapunov block form
 \[A_{\eta}(\tilde{x})=\begin{pmatrix}
\ A^1_{\eta}(x_{0})\\
&\ A^2_{\eta}(x_{0})  
\end{pmatrix}.\]

For $\tilde{x}=\{x_n\}_{n\in\mathbb{Z}}$ and $ r>0$, let $T_{x_{0}}M(r):=\{w\in T_{x_{0}}M|||w||\leq r\}$; choose $r(\tilde{x})$   small enough so that for every $x_{0}\in M$ the exponential map $\exp_{x_{0}}:T_{x_{0}}M(r(\tilde{x}))\rightarrow M$ is an embedding, $||D_w\exp_{x_{0}}||\leq 2$,  $\exp_{f(x_{0})}$ is injective on $\exp^{-1}_{f(x_{0})}\circ f \circ \exp_{x_{0}}(T_{x_{0}}M(r(\tilde{x})))$.  As the domains of exponential maps have uniform radius, we can choose $r(\tilde{x})= r.$ Define 
$$f_{\tilde{x}}:=C_{\eta}(\tilde{f}(\tilde{x}))\circ \exp^{-1}_{x_{1}}\circ f\circ \exp_{x_0}\circ C_{\eta}^{-1}(\tilde{x}),$$
so $f_{\tilde{x}}$ is defined in the ellipsoid $P(x_{0})=C_{\eta}(\tilde{x})(T_{x_{0}}M(r))\subset \mathbb{R}^d$. Let $p(\tilde{x})=r min\{||C_{\eta}(\tilde{x})||,||C_{\eta}(\tilde{x})^{-1}||\}$; thus if $w\in T_{x_{0}}M$ and $||w||\leq p(\tilde{x})$ then $w\in P(x_{0})$, that is the Euclidean ball $B(0,p(\tilde{x}))$ is contained in $P(x_{0})$.

Now write $f_{\tilde{x}}(w)=A_{\eta}(\tilde{x})w+h_{\tilde{x}}(w)$ and observe that $D_0f_{x_{0}}=A_{\eta}(\tilde{x})$, by the chain rule, so $D_0h_{\tilde{x}}=0$. Write $\exp^{-1}_{x_{1}}\circ f\circ \exp_{x_{0}}=Df_{x_{0}}+g_{x_{0}}$. Since $f\in \mathcal{C}^{1+\alpha}$, there exists $L>0$ such that $||D_ug_{x_{0}}||\leq L||u||^{\alpha},$ thus 
$$||D_wh_{\tilde{x}}||=||D_w(C_{\eta}(\tilde{f}(\tilde{x}))\circ g_{\tilde{x}}\circ C^{-1}_{\eta}(\tilde{x}))||\leq L ||C_{\eta}(\tilde{f}(\tilde{x}))||||C^{-1}_{\eta}(\tilde{x})||^{1+\alpha}||w||^{\alpha}.$$
Hence if $||w||$ is sufficiently small the construction of the nonlinear part of $f_{\tilde{x}}$ is negligible. In particular,
$$||D_wh_{\tilde{x}}||<\eta\text{ for }||w||<\delta_{\eta}(\tilde{x}):=(L||C_{\eta}(\tilde{f}(\tilde{x}))||||C^{-1}_{\eta}(\tilde{x})||^{1+\alpha}/\eta)^{-1/\alpha}.$$
By the Mean Value Theorem also
$$||h_{\tilde{x}}w||<\eta\text{ for } ||w||<\delta_{\eta}(\tilde{x}).$$
From the definition of $\delta_{\eta}(\tilde{x})$ we have 
$$\lim_{m\rightarrow \infty}\frac{1}{m}\log \delta_{\eta}(\tilde{f}^m(\tilde{x}))=\lim_{m\rightarrow}\frac{1}{\alpha m}\log||C_{\eta}(\tilde{f}^m(\tilde{x}))||+\lim_{m\rightarrow}\frac{1+\alpha}{\alpha m}\log||C_{\epsilon}^{-1}(\tilde{f}^m(\tilde{x}))||=0.$$
Applying the Tempering-Kernel Lemma \ref{Tempering} to $\delta_{\eta}(\tilde{x})$ we find a measurable $K_{\eta}:\tilde{\Lambda}\rightarrow \mathbb{R}^{+}$ such
that $K_{\eta}(\tilde{x})\geq \delta_{\eta}^{-1}(\tilde{x})$ and $e^{-\eta}\leq K_{\epsilon}(\tilde{x})/K_{\eta}(\tilde{f}(\tilde{x}))\leq e^{\eta}$.
Define $\xi_{\eta}(\tilde{x})=K_{\eta}(\tilde{x})^{-1}\leq \delta_{\eta}(\tilde{x}).$ Then
$$e^{-\eta}\leq \xi_{\eta}(\tilde{x})/\xi_{\eta}(\tilde{f}(\tilde{x}))\leq e^{\eta}.$$
Now $$\phi_{\tilde{x}}:B(x_{0}, K^{-1}\xi_{\eta}(\tilde{x}))\rightarrow \mathbb{R}^d, z\mapsto  C_{\eta}(\tilde{x})\circ\exp^{-1}_{x_{0}}(z)$$ is obviously an embedding. Condition $(a)$ follows from the continuous property of $C_{\eta}$ on $\tilde{\Lambda}(\eta)$ and condition (b) follows from the definition. We only need to prove (c) and (d). Actually, on the set $\phi_{\tilde{x}}(B(x_{0}, K^{-1}\xi_{\eta}(\tilde{x})))$ , we have 
\begin{eqnarray*}
&&d(\phi_{\tilde{f}\tilde{x}}\circ f\circ \phi^{-1}_{\tilde{x}}, C_{\eta}(\tilde{f}(\tilde{x}))\circ Df(x_{0})\circ C^{-1}_{\eta}(\tilde{x}))\\
&=&d(C_{\eta}(\tilde{f}\tilde{x})\circ\exp^{-1}_{x_{1}}\circ f \circ\exp_{x_0}\circ C^{-1}_{\eta}(\tilde{x}), C_{\eta}(\tilde{f}(\tilde{x}))\circ Df(x_{1})\circ C^{-1}_{\eta}(\tilde{x}))\\
&=&||f_{(\tilde{x})}-A_{\eta}(\tilde{x})||\\
&\leq&||h_{\tilde{x}}||\\
&\leq&\eta.
\end{eqnarray*}
Now we give the proof of condition (d). From the definition of $C_{\eta}$, we only need to proof the second inequality. For all $\tilde{x}\in\tilde{\Lambda}(\eta)$ and for any $y,y'\in B(x_{n}, K^{-1}\xi_{\eta}(\tilde{f}^{n+1}\tilde{x}))$, we have 
\begin{eqnarray*}&&||\phi_{\tilde{f}^n\tilde{x}}(y)-\phi_{\tilde{f}^n\tilde{x}}(y')||\\
&\leq&||C_{\eta}(\tilde{f}^n\tilde{x})||||y-y'||\\
&\leq&K \exp(\eta|n|)||y-y'||.
\end{eqnarray*}
Actually, the last inequality can be improved to the form stated in our theorem. This is a direct result from the continuity of $C_{\eta}.$

\end{proof}
The set $B(x_n,K^{-1}\xi(\tilde{f}^n\tilde{x}))$ are called Lyapunov neighborhoods of the orbit $\tilde{x}$ at $x_n$. Although it may happen that the restriction of  $f$ on Lyapunov neighborhoods may not be invertible, the degeneration happens only along the stable direction which does not affect shadowing mechanism.  The size of Lyapunov neighborhoods decay slowly (at rate at most $e^{-\eta}$) for points along orbit $\tilde{x}$. 
\begin{Def}A sequence $(\tilde{x}_n)_{n\in\mathbb{Z}}\subset M^f$ is called an $\varepsilon-$pseudo orbit with jumps in set $\tilde{\Lambda}(\eta)$ 
if $ d(\tilde{f}(\tilde{x}_n), \tilde{x}_{n+1})<\varepsilon,$ and $ d(\tilde{f}(\tilde{x}_n), \tilde{x}_{n+1})>0 \Longrightarrow 
\tilde{f}(\tilde{x}_n),\tilde{x}_{n+1}\in\tilde{\Lambda}(\eta) ,$ for every $n\in\mathbb{Z}$.
\end{Def}
\begin{Thm}\label{shadowing}{\bf (Orbit Shadowing property for mappings )}Let  $f$ be a $\mathcal{C}^{1+\alpha}$ mapping with ergodic hyperbolic measure $\mu$ and satisfying the integrability condition. Then for every $\eta>0$ sufficiently small and $\tilde{\Lambda}(\eta)$ is a uniformity block  of tolerance $\eta$ for $\mu$,
 there exist $C>0,\varepsilon_0>0$ with the following properties holding:
 
For every $\varepsilon\in(0,\varepsilon_0),$ if $(\tilde{x}_n)_{n\in\mathbb{Z}}$ is an $\varepsilon-$pseudo orbit for $\tilde{f}$ with jumps
in $\tilde{\Lambda}(\eta)$, then there exists a unique orbit $\tilde{y}\in\tilde{\Lambda}$,
 $C\varepsilon$- shadowing  $(\tilde{x}_n)_{n\in\mathbb{Z}}$ , i.e. $d(\tilde{f}^n\tilde{y},\tilde{x}_n)\leq C\epsilon$, for all $n\in\mathbb{Z}$. 
 \end{Thm}
\begin{proof} 

Our goal here is to construct sequences of mappings on space $\mathbb{R}^d$ satisfying  conditions in  Lemma \ref{Shadowing}. For any  $(\tilde{x}_n)_{n\in\mathbb{Z}}$, an $\varepsilon$ pseudo orbit for $\tilde{f}$ with jumps
in $\tilde{\Lambda}(\eta)$, denote $x_n=\pi(\tilde{x}_n).$ $0<\varepsilon<\varepsilon_0$ is determined later.  Following the notation in the proof of Theorem \ref{Pesin}, we have sequences of  mappings $\tilde{g}_n:\mathbb{R}^d\rightarrow \mathbb{R}^d, \tilde{g}_n=\phi_{\tilde{f}\tilde{x}_n}\circ f\circ \phi^{-1}_{\tilde{x}_n}:\mathbb{R}^d\rightarrow \mathbb{R}^d$ and sequences of linear mappings $L_n:\mathbb{R}^d\rightarrow \mathbb{R}^d, L_n=C(\tilde{f}\tilde{x}_n\circ Df(\pi(\tilde{x}_n))\circ C^{-1}(\tilde{x}_n)= A_{\eta}(\tilde{x})=\begin{pmatrix}
\ A^1_{\eta}(\tilde{x}_n)\\
&\ A^2_{\eta}(\tilde{x}_n)  
\end{pmatrix}.
$  From Theorem \ref{Pesin}, we have $d(\tilde{g}_n(v),L_n(v))\leq \eta, \forall n\in \mathbb{Z}, $  and any $v\in \phi_{\tilde{x}}(B(x_n,K^{-1}\xi(\tilde{x}_n))).$
Denote $$\Phi_n:\mathbb{R}^d\rightarrow \mathbb{R}^d=\phi_{\tilde{x}_{n+1}}\circ f\circ \phi^{-1}_{\tilde{x}_n}=L_n+\phi_{\tilde{x}_{n+1}}\circ f\circ \phi^{-1}_{\tilde{x}_n}-L_n.$$Considering the jumping points, i.e. $0<d(\tilde{f}\tilde{x}_n,\tilde{x}_{n+1})\leq \varepsilon, $ by the continuity of $\phi_{\tilde{x}}$ on $\tilde{\Lambda}(\eta)$ and $\phi\in\mathcal{C}^{\infty}$, we have 
\begin{eqnarray*}
&&|\phi_{\tilde{x}_{n+1}}\circ f\circ \phi^{-1}_{\tilde{x}_n}(0)-L_n(0)|\\
&=&|\phi_{\tilde{x}_{n+1}}(f(x_n))|\\
&=&|\phi_{\tilde{x}_{n+1}}(f(x_n))-\phi_{\tilde{f}(\tilde{x}_{n})}(f(x_n))|\\
&\leq&K\varepsilon.
\end{eqnarray*}
Moreover,
\begin{eqnarray*}
&&|\phi_{\tilde{x}_{n+1}}\circ f\circ \phi^{-1}_{\tilde{x}_n}(v)-L_n(v)-\phi_{\tilde{x}_{n+1}}\circ f\circ \phi^{-1}_{\tilde{x}_n}(v')+L_n(v')|\\
&\leq&|L_n(v)-L_n(v')|+|\phi_{\tilde{x}_{n+1}}\circ f\circ\phi_{\tilde{x}_n}(v)-\phi_{\tilde{x}_{n+1}}\circ f\circ\phi_{\tilde{x}_n}(v')|\\
&\leq&C|v-v'|^{\alpha}
\end{eqnarray*}
It is easy to see that sequences  of mappings $\Phi_n$ also satisfy other conditions in Lemma \ref{Shadowing}, thus there is an unique sequence of points $z_n\in\mathbb{R}^d$  satisfying $\Phi_n(z_n)=(z_{n+1})$. Then, $\{\phi_{\tilde{x}_n}^{-1}z_n\}_{n\in\mathbb{Z}}$ is a real orbit under map $f$. 



\end{proof}
 
As $\tilde{f}:M^f\rightarrow M^f$ is invertible and measure entropy comes from any positive measure subset,
invariant or not,  we have the following type of Katok's argument which says there exist horseshoes in inverse limit space $M^f$.
 In order to clarify to process of projecting 
the horseshoe in the inverse limit space to the initial space, we give the construction here briefly. Our proof follows the the idea used in \cite{ACW}. These ideas of constructing pseudo orbits also appeared in \cite{LLST, LST,LSY}.
\begin{Thm}\label{Kat}{\bf(Katok's argument for mappings )}
Let $f:M\rightarrow M$ be any $\mathcal{C}^{1+\alpha}$ mapping  on a closed Riemannian manifold  $M$ with an ergodic hyperbolic probability measure $\mu$.   Set any small constant $\delta>0$ and a weak $*$ neighborhood $\tilde{\mathcal{V}}$ of $\tilde{\mu}$ in the space of $\tilde{f}$-invariant probability measures on $\tilde{\Lambda}$. Then there exists a horseshoe $\tilde{H}\subset M^f$ such that:
\begin{enumerate}
 \item $h_{top}(\tilde{H},\tilde{f})>h(\tilde{\mu},\tilde{f})-\delta=h(\mu,f)-\delta$. 
\item if $\tilde{\lambda}_1>\tilde{\lambda}_2>\cdots>\tilde{\lambda}_k$ are the distinct Lyapunov exponents of $\mu$, with 
multiplicities $n_1,\cdots,n_k\geq 1$, denote $\tilde{\lambda}$ the same as before, then there exists a dominated splitting on $T_{\tilde{x}}M=\sqcup T_{\pi(\tilde{f}^n\tilde{x})}M, \tilde{x}\in \tilde{H}$:
$$T_{\tilde{x}}M=E^u\oplus E^s,$$ and  there exists $N\geq1$ such that for each $i=1,2$ each $\tilde{x}\in\tilde{H}$ and each unit vector $v\in E^u(\pi(\tilde{x})), u\in E^s(\pi(\tilde{x}))$,
$$||Df^{-N}_{\pi(\tilde{x})}(v)||\leq \exp((-\tilde{\lambda}_i+\delta)N),$$$$||Df^{N}_{\pi(\tilde{x})}(u)||\leq \exp((-\tilde{\lambda}_i+\delta)N)$$
\item all the invariant probability measures supported on $\tilde{H}$ lies in $\mathcal{V}.$

\item $\tilde{H}$ is $\delta-$close to the support of $\tilde{\mu}$  in the Hausdorff distance.
\end{enumerate}
\end{Thm}
\begin{proof}
{\bf Step 1:} For tha sake of estimating the distance of measures, we first give a finite collection of continuous functions on $\tilde{\Lambda}$ that can be used in weak $*$ metric. Let $C(\tilde{\Lambda})$ be the space of real valued continuous functions defined on $\tilde{\Lambda}$. Choose $\gamma\in(0,\delta)$ and let $\tilde{\phi}_1,\tilde{\phi}_2,\cdots,\tilde{\phi}_k$ be a finite collection of $C(\tilde{\Lambda})$ such that $\tilde{\mathcal{V}}$ contains the set of probability measures $ \tilde{\nu}$ satisfying:
$$\sum_{i=1}^{i=k}\frac{|\int_{\tilde{\Lambda}}\tilde{\phi}_i d\tilde{\mu}-\int_{\tilde{\Lambda}}\tilde{\phi}_i d\tilde{\nu}|}{2^i}<\gamma.$$

{\bf Step 2:}   
By Theorem 1.1 in \cite{Kat}, measure-theoretic entropy is the exponential growth rate of the minimal number 
of Bowen ball covering a positive measure set.  More specifically,
given $x \in \tilde{\Lambda}, \rho > 0, n\in \mathbb{N}$, the $(n,\rho)$ Bowen ball is defined as 
$$B(\tilde{x},n,\rho)=\{\tilde{y}\in \tilde{\Lambda} \,\,| \,\,d(\tilde{f}^i(\tilde{x}),\tilde{f}^i(\tilde{y}))\leq\rho,0\leq i\leq n-1\}.$$
Define 
$$N(n,\rho,\xi)=\min\sharp\{\tilde{x}_1,\cdots,\tilde{x}_k|\mu(\cup_{i}B(\tilde{x},n,\rho))>1-\xi\}.$$
For any positive number $\xi<1$,
measure entropy (which is independent of $\xi$) is defined as 
$$h_{\tilde{\mu}}(\tilde{f})=\limsup_{n\rightarrow \infty}\frac{1}{n}\log N(n,\rho,\xi) .$$
We also define $(n,\rho)$-separated set here. $S(n,\rho)$ is called a $(n,\rho)$-separated set for set $K$ if for any point $\tilde{x}\in K\subset \tilde{\Lambda}$, there exists a point $\tilde{y}\in S(n,\rho)$ such that $d(\tilde{f}^i(x),\tilde{f}^j(x))\geq \rho,$ for some $i\in [0,n-1].$
We will take advantage of the fact that the maximal number of $(n,\rho)$-separated set is bigger than the 
minimal number of  $(n,\rho)$ Bowen balls covering the same set.   

{\bf Step 3:}
Now we are ready to give the scale of Bowen balls and the scale of separation, i.e. $\rho$ in the definition of the Bowen ball and the separated set.  
Assume $0<\varepsilon_1<\min(\frac{\gamma}{2h_{\tilde{\mu}}\tilde{f}+4},\frac{\delta}{4})$. Choose $\rho>0$ small enough and $N_0\in \mathbb{N}$  such that for any $n\geq N_0$
$$ N(n,\rho,\tilde{\mu}(\tilde{\Lambda}(\eta))/2)> e^{n(h_{\tilde{\mu}}(\tilde{f})-\varepsilon_1)}$$
and such that for any $d(x,y)\leq \rho,\forall x,y\in \tilde{\Lambda}$, $$|\tilde{\phi}_i(x)-\tilde{\phi}_i(y)|\leq \frac{\gamma}{2}, 1\leq i\leq k.$$
The small number $\rho$ here is a separation scale.

{\bf Step 4:}
We next give a shadowing scale which is smaller than the separation scale $\rho$ given above. We also filter  recurrence points  for   $\tilde{\Lambda}(\eta)$.
Fix $0<\varepsilon_2<\frac{\rho}{4C}$, where $C$ is the Lipschitz constant given in Theorem \ref{shadowing}. 
Let $\mathcal{U}=\{B(\tilde{x}_i,\varepsilon_2)\,\,|\,\, \tilde{x}_i\in\tilde{\Lambda}(\eta), 1\leq i\leq t\}$ be a cover of $\tilde{\Lambda}(\eta)$. 

 For any $n\in\mathbb{N}$, let
\begin{eqnarray*}
\tilde{\Lambda}'(\eta)_n&=&\{\tilde{x}\in\tilde{\Lambda}(\eta):\exists i\in[n,(1+\varepsilon_2)n] \\
&& s.t.\,\, \tilde{x},\tilde{f}^i\tilde{x}\in B(\tilde{x}_k,
\varepsilon_2), for\,\, some \,\,k\in[1,t] \}.
\end{eqnarray*} 
By Poincare recurrence theorem, $\tilde{\mu}(\tilde{\Lambda}'(\eta)_n)\rightarrow \tilde{\mu}(\tilde{\Lambda}(\eta))$, as $n\rightarrow +\infty$. 
Set $$\tilde{\Lambda}(\eta)_n=\{\tilde{x}\in\tilde{\Lambda}'(\eta)_n: sup_{l\geq n}\max_{1\leq i\leq k}|\frac{1}{l}\sum^{l}_{j=1}\tilde{\phi}_i(\tilde{f}^j(\tilde{x}))-\int_{\tilde{\Lambda}}\phi_i d _{\tilde{\mu}}|<\frac{\gamma}{2}\}. $$
 Birkhoff Ergodic Theorem implies that $\tilde{\mu}(\tilde{\Lambda}(\eta)_n)\rightarrow \tilde{\mu}( \tilde{\Lambda}(\eta))$, as $n\rightarrow \infty.$

{\bf Step 5:} 
We  chose $(n,\rho)$-separated set covering $\tilde{\Lambda}(\eta)$ in this step.  We need not only separation property but also a common return time to $\tilde{\Lambda}(\eta)$, so that we can control the segments.   We  also use some combinatory techniques to estimate the number of $(n,\rho)$-separated segments with the common return time to $\tilde{\Lambda}(\eta)$. 

For each $n\in\mathbb{N}$, let $S(n,\rho)$ be a maximal $(n,\rho)$-separated set in $\tilde{\Lambda}(\eta)_n$. Without of loss of generality, we can assume that each two points in $S(n,\rho)$ come from different orbits (if there are two points in the same orbit, just give a small perturb of it).  
 Then,
$$\tilde{\Lambda}(\eta)_n\subset\cup_{\tilde{x}\in S(n,\rho)}B(\tilde{x},n,2\rho)$$
and for $N_1$ large enough such that for any $n\geq N_1$
we get $$\sharp S(n,\rho)\geq N(n,2\rho, \tilde{\mu}(\tilde{\Lambda})/2)\geq e^{n(h_{\tilde{\mu}}(\tilde{f})-\varepsilon_1)}. $$ 
For $n\in[N_1,(1+\varepsilon_2)N_1]$, let 
$$V_n=S(N_1,\rho)\cap \{\tilde{x}\in \tilde{\Lambda}(\eta)| \tilde{x}, \tilde{f}^n(\tilde{x})\in B(\tilde{x}_k,\varepsilon_2), for\, some\, k\in[1,t]\}$$
and let $N\in[N_1,(1+\varepsilon)N_1]$ be the value of $n$ maximizing $\sharp V_n$. Assuming $N_1$ large enough,
$$\sharp V_N\geq\frac{S(N_1,\rho)}{\varepsilon N_1}\geq e^{N(h_{\tilde{\mu}}(\tilde{f})-2\varepsilon_1)}.$$

{\bf Step 6:}
Now we have had lots of separated segments with common return time. But in order to construct  pseudo orbits from the segments of these separated points, we need to chose separated segments coming in to and getting away from the same ball.

Chose $k\in[1,t]$ such that $B(\tilde{x}_k,\varepsilon_2)\cap V_n$ has maximal cardinality and let
 $\tilde{Y}=\{\tilde{y}_1,\cdots,\tilde{y}_l\}=B(\tilde{x}_k,\varepsilon_2)\cap V_N$. From the choice of $\tilde{Y}$, it is natural that $l$ is large with respect to entropy, i.e.
$$l\geq \frac{\sharp V_N}{t}\geq \frac{1}{t}e^{N(h_{\tilde{\mu}}(\tilde{f})-2\varepsilon_1)}.$$

{\bf Step 7:} Now we can construct pseudo orbits.
Consider the set of all orbits whose segments of length $N$  originate in $\tilde{Y}$ and end in $B$. Concatenating these strings defines a 
two sided shift $\sigma_l$ based on $l$ symbols, which has topological entropy $\log l\geq N(h(\tilde{\mu},\tilde{f})-2\xi)-\log t.$ We will construct a horseshoe $\tilde{H}
\subset M^f$ such that $\tilde{f}^N|_{\tilde{H}}$ has $\sigma_l$ as a topological factor. Considering the set $\tilde{\mathcal{Y}}$ of all $\varepsilon_2-$pseudo orbits of the form:
$$\cdots\tilde{y}_{i_{-1}}, \cdots, \tilde{f}^{N-1}(\tilde{y}_{i_{-1}}),\tilde{y}_{i_{0}}, \cdots, \tilde{f}^{N-1}(\tilde{y}_{i_{0}}),\tilde{y}_{i_{1}}, \cdots, \tilde{f}^{N-1}(\tilde{y}_{i_{1}}),\cdots$$
where $\tilde{y}_{i_j}\neq\tilde{y}_{i_{j+1}}\in \tilde{\mathcal{Y}}.$  Note that these are also $\varepsilon_2-$pseudo orbits with jumps in $\tilde{\Lambda}(\eta),$ since $\tilde{f}^N(\tilde{y})\in \tilde{\Lambda}(\eta)$, for all $\tilde{y}\in\tilde{\mathcal{Y}}$. Each element of $\tilde{\mathcal{Y}}$ can be naturally encoded as an element of $\{1,\cdots,l\}^{\mathbb{Z}}\times\{0,\cdots, N-1\}$. 
We define $\tilde{H}$ to be the set of $\tilde{x}\in M^f$ whose $\tilde{f}-$orbit $C\varepsilon_2-$shadowing some pseudo orbits in $\tilde{\mathcal{Y}}$. 

Denote a  Markovian subshift $\mathcal{C}$ of shift $\{1,\cdots,l\}^{\mathbb{Z}}\times\{0,\cdots, N-1\}$ with the Markovian graph derived from the graph of $\{1,\cdots,l\}^{\mathbb{Z}}\times\{0,\cdots, N-1\}$ by dropping the chains from one vertex to itself.  
 Theorem \ref{shadowing} and $\rho>\frac{C\varepsilon_2}{4}$ imply there is a continuous bijection between $\tilde{H}$ and $\mathcal{C}$. 

Hyperbolicity of $\tilde{H}$ follows from the fact that the orbit of any $\tilde{x}\in \tilde{H}$ stays in the union of finitely many regular neighborhoods, on which $f$ stays close to a uniformly hyperbolic sequence in $\{A_{\epsilon}(\tilde{x}_{j_1}), A_{\epsilon}(\tilde{x}_{j_2}),\cdots,A_{\epsilon}(\tilde{x}_{j_N})\}$. Other conclusions follows from our construction directly.  Thus we finish the proof. 
\end{proof}
\begin{Cor}\label{Growth}Let $f:M\rightarrow M$ be a  $\mathcal{C}^{1+\alpha}$ mapping   preserving an ergodic
hyperbolic probability measure $\mu$.  We have 
$$\limsup_{n\rightarrow+\infty}\frac{1}{n}\log P_n(f)\geq h_{\mu}(f)$$
where $P_n(f)$ denotes the number of periodic points with  period $n$.
\end{Cor}
\begin{proof}For symbolic systems $(\sigma,\Sigma_{l})$
$$h_{top}(\Sigma_l,\sigma)=\frac{1}{m}\log \sharp\{x\in\Sigma_{l}|\sigma^m(x)=x\}.$$
Then, by Theorem (\ref{Kat}), the corollary follows. 
\end{proof}


\begin{Def}\label{PH}For map $f:M\rightarrow M$ and any constants $K>0, a>0$, we say periodic point $p$ with period $P(p)$ has $(K,a)$-hyperbolicity if there exists an invariant splitting 
$T_{f^i(p)}M=E^s_{f^i(p)}\oplus E^u_{f^i(p)}, 0\leq i\leq P(p)-1$
along orbit $\{f^i(p)\}_{i=0}^{P(p)-1}$ such that 
$$||Df^j_{f^i(p)(v)}||\leq Ke^{-ja}||v||, \forall v\in E^s_{f^i(p)}$$
and  $$||Df^j_{f^i(p)(v)}||\geq Ke^{ja}||v||, \forall v\in E^u_{f^i(p)}.$$
Let $PH(n,f,K,a)$ be  the collection of periodic points with period $n$ and uniform $(K,a)$-hyperbolicity.
\end{Def}
\begin{Thm} \label{Growth2} For any  $\mathcal{C}^{1+\alpha}$ mapping $f:M\rightarrow M$,  we have 
$$\limsup_{n\rightarrow+\infty}\frac{1}{n}\log P_n(f)\geq H(f).$$
 Moreover,   we have $$\lim_{a\rightarrow 0^{+}}\lim_{K\rightarrow 0^{+}}\limsup_{n\rightarrow \infty}\frac{1}{n}\log \sharp PH(n,f, K,a)=H(f).$$
\end{Thm}
\begin{proof}The first part is a direct corollary.  We only need to prove the second equality. 
As periodic points for $f$ can also be viewed as periodic points for $\tilde{f}$, we use the same notation. Assume the Lyapunov splitting over the orbit of  periodic point $p$ with period $P(p)$ is 
$$T_{f^i(p)}M=E^s_{f^i(p)}\oplus E^u_{f^i(p)}, \forall 0\leq i\leq P(p)-1. $$
Let $I(p)$ be the index of $p$, i.e. $I(p)$ is the dimension of stable bundle $E^s$. 
We define the following collection of periodic points with uniform hyperbolicity and the same index,
\begin{eqnarray*}PH(n,f,a, K, I)&=&\{p\in P_n(f)| ||Df^{i}_{f^j(p))}v||\geq  K e^{ia}||v||, \forall v\in E^u_{f^j(p)},\\
&&||Df^{i}_{f^j(p))}u||\leq K^{-1}e^{-ia}||u||, \forall u\in E^s_{f^j(p)},  \\
&&\forall 0\leq j\leq n-1, I(p)=I\}.
\end{eqnarray*}
As the splitting on periodic points can be continuously extended to the closure set, $\tilde{f}|(\overline{\cup_nPH(n,f,a,K,I)})$ is uniformly hyperbolic. 
 Thus $\tilde{f}|\overline{\cup_nPH(n,f,a,K, I)}$ is an expansive (from unique shadowing property) homoeomorphism and then $f|\pi(\overline{\cup_nPH(n,f,a,K,I)})$ is an expansive map. From the fact that  $\pi(PH(n,f,a,K, I))$ is a n-separated  set, one has $$ \lim_{n\rightarrow+\infty}\frac{1}{n}\log \sharp PH(n,f,a,K,I) \leq h(f|\overline{\pi(\cup_nPH(n,f,a,K, I)))}).$$   From the principle of variation, for any $\varepsilon>0$, there exists a hyperbolic measure $\mu$ supported on $\overline{\pi(\cup_nPH(n,f,a,K,I))}$ such that $h(f|\overline{\pi(\cup_nPH(n,f,a,K,I))})\leq h_{\mu}(f)+\varepsilon\leq H(f)+2\varepsilon. $ From the arbitrary choice of $\varepsilon$, we obtain 
 $$\limsup_{n\rightarrow+\infty}\frac{1}{n}\log \sharp PH(n,f,K,a,I)\leq H(f).$$ Since $$PH(n,f,K,a)=\cup_{0\leq I\leq d}\cup_{K>0}PH(n,f,a,K,I),$$ 
 we have $$\lim_{k\rightarrow 0^{+}}\lim_{n\rightarrow \infty}\frac{1}{n}\log \sharp PH(n,f,K,a)\leq H(f). $$
 On the other hand, 
 by Theorem \ref{Kat},  we have 
 $$\lim_{a\rightarrow 0^{+}}\lim_{K\rightarrow 0^{+}}\lim_{n\rightarrow \infty}\frac{1}{n}\log \sharp PH(n,f,K,a)\geq H(f),$$
 and then
 $$\lim_{a\rightarrow 0^{+}}\lim_{K\rightarrow 0^{+}}\lim_{n\rightarrow \infty}\frac{1}{n}\log \sharp PH(n,f,K,a)=H(f). $$
  \end{proof}

\section{Relation between exponential growth rate of periodic points and degree}
Asymptotic growth rate of  the complexity of the orbit structure attracts people's attention for a long time. There are several points of view to describe asymptotic behaviors of dynamical systems, such as topology, measure, homology, etc.  Commonly, one cares about the growth rate of the number of  periodic points, measure-theoretic entropy, topological entropy and the spectral radii of the action on homology, etc. 

Let $M$ be a compact connected $d-$dimensional manifold. For $\mathcal{C}^1$ mapping $f:M\rightarrow M$,   Misiurewicz and Pryztycki \cite{MP} proved that \[h_{top}(f)\geq \log |deg(f)|.\] 
Let $P_n(f)$ denotes the number of periodic points with period $n$.  Katok \cite{Kat} proved that for any $\mathcal{C}^{\infty}$ surface diffeomorphism $f:M\rightarrow M$ we have 
\[\limsup_{n\rightarrow+\infty}\frac{1}{n}\log P_n(f)\geq h_{top}(f). \]
Inspired by these two results,  Shub posed an  interesting case in the problem 3 of \cite{Shub3}.  Let $f$  be a smooth degree two $C^{1+\alpha}$ map on  2-shpere $S^2$ where $\alpha>0$. 

\smallskip
{\bf Problem (Shub) :} Does $\limsup_{n\rightarrow +\infty}\frac{1}{n}\log P_n(f)\geq \log2$ hold?
\smallskip

In order to get periodic points, a usual technique is the closing lemma used by Katok in \cite{Kat} which is based on the hyperbolicity of invariant measures. If we assume there exists a hyperbolic invariant ergodic measure $\mu$ of $\mathcal{C}^{1+\alpha}$ map $f$ with $h_{\mu}(f)\geq \log deg(f)$, then  from Corollary \ref{Growth} we get 
$\limsup_{n\rightarrow +\infty}\frac{1}{n}\log P_n(f)\geq \log \text{deg}(f).$
 There is also a direct corollary from  Corollory \ref{Growth2} as follows. 
\begin{Cor} \label{Cor}Let $M$ be a compact connected $d-$dimensional manifold. For any $\mathcal{C}^{1+\alpha}$ mapping $f:M\rightarrow M$, if   $H(f)\geq \log \text{deg}(f)$. Then we have 
$$\limsup_{n\rightarrow+\infty}\frac{1}{n}\log P_n(f)\geq \log \text{deg}(f).$$
\end{Cor}

We give some notes for the  case when $M$ is a surface.   It is easy to see that if $f$ is a diffeomorphism, then every invariant measure with positive entropy is hyperbolic. But this might not be true for endmomorphisms. For noninvertible mappings on surface, one might can not get hyperbolic invariant measures with  measure-theoretic entropy approximating topological entropy.  It may usually happen that for noninvertible mapping $f$,  $$H(f)<h_{top}(f), $$ 
such as  examples given by  Pugh and Shub in their new paper \cite{Shub4}. In other words, from the equality in Theorem \ref{Growth2}, it is highly possible that  the growth rate  of  the number of saddle periodic points is strictly smaller than topological entropy.   But it is still possible that the growth rate of  the number of periodic points with zero Lyapunov exponents is greater than degree.  One  may need some topological techniques to get Shub's question.

\end{document}